\documentclass[11pt]{amsart}

\usepackage{graphicx,epsfig}
\usepackage[dvipsnames]{xcolor}
\usepackage{amsfonts,amsmath,amssymb,amsthm}
\usepackage[normalem]{ulem}
\usepackage{comment,constants}
\usepackage{hyperref}
\usepackage{mathtools}
\usepackage{multirow}

\newtheorem{thm}{Theorem}[section]
\newtheorem{prop}[thm]{Proposition}
\newtheorem{lem}[thm]{Lemma}
\newtheorem{cor}[thm]{Corollary}
\newtheorem{conj}[thm]{Conjecture}
\newtheorem*{asm*}{Assumptions}
\newtheorem{asm}{Assumption}

\theoremstyle{remark}
\newtheorem{rem}[thm]{Remark}
\newtheorem*{rem*}{Remark}

\theoremstyle{definition}

\newcommand{\ra}{\rightarrow}

\newcommand{\R}{\mathbb R}     % For Real numbers
\newcommand{\Z}{\mathbb Z}     % For Integers
\renewcommand{\a}{\alpha}

\renewcommand{\d}{\delta}

\newcommand{\e}{\varepsilon}

\newcommand{\s}{\sigma}

\renewcommand{\epsilon}{\varepsilon}
\renewcommand{\k}{\kappa}
\renewcommand{\t}{\tau}
\renewcommand{\r}{\rho}

\newcommand{\bigo}{\mathcal{O}}

\newcommand{\fl}[1]{\lfloor #1 \rfloor}  % Floor function
\newcommand{\ind}[1]{ \mathbf{1}_{ \{ #1 \} } } % Indicator functions

\newconstantfamily{c}{symbol=c}
\newconstantfamily{CU}{symbol=C_U}

\newcommand{\w}{\omega}              % Shortcut for \omega
\renewcommand{\P}{\mathbb{P}}        % Annealed Probability
\newcommand{\E}{\mathbb{E}}          % Annealed Expectation
\newcommand{\vp}{\mathrm{v}_0}       % Limiting velocity
\DeclareMathOperator{\Var}{Var}
\DeclareMathOperator{\Cov}{Cov}

\title[Optimal quenched CLT rates of convergence]
{Optimal rates of convergence in quenched central limit theorem for hitting times of random walks in random environments}
\author{Sung Won Ahn}
\address{Sung Won Ahn\\Roosevelt University\\Department of Mathematics and Actuarial Science\\430 S. Michigan Ave.\\Chicago, IL 60605\\USA} % add more detailed address
\email{sahn02@roosevelt.edu}
\author{Jonathon Peterson}
\address{Jonathon Peterson\\Purdue University\\Department of Mathematics\\150 N University Street\\West Lafayette, IN  47907\\USA}
\email{peterson@purdue.edu}
\urladdr{http://www.math.purdue.edu/~peterson}
\thanks{J. Peterson was partially supported by the Simons Foundation (Award \#635064).}


\begin{document}

\begin{abstract}
We consider the rates of convergence of the quenched central limit theorem for hitting times of one-dimensional random walks in a random environment. 
Previous results had identified polynomial upper bounds for the rates of decay which are sometimes slower than $n^{-1/2}$ (the optimal rate in the classical Berry-Esseen estimates). Here we prove that the previous upper bounds are in fact the best possible polynomial rates for the quenched CLT. 
\end{abstract}

\maketitle

\section{Introduction}

Random walks in random environments (RWRE) are a simple model for random motion in a non-homogeneous medium. Despite the simplicity of the model, RWRE can exhibit a rich array of limiting behaviors not seen in the classical model of simple random walks. 
In particular, the limiting distributions of RWRE can in some cases be non-Gaussian and/or the scaling can be non-diffusive \cite{kksStable,sRecurrent}. 
However, in certain cases one does have central limit theorem (CLT) like behavior; that is Gaussian limiting distributions with diffusive scaling \cite{kksStable,lBCLT,sSlowdown,bszCPDRW,gQCLT,pThesis,bzQCLT,rsQCLT}. 
Recently there has been interest in studying the rate of convergence of these limit theorems by obtaining quantitative bounds analagous to the Berry-Esseen bounds for the classical CLT \cite{mQCLTRWRC,apQCLTrates,gpBERP,anBEQH}.
These results have all focused on obtaining quantitative \emph{upper bounds} on the error in the CLT. 
When these upper bounds are $\bigo(n^{-1/2})$ (like the Berry-Esseen bounds for the classical CLT) it is easy to see they are asymptotically optimal in the sense that one can also obtain a lower bound of the form $c n^{-1/2}$ for some $c>0$. If the bounds are larger than $n^{-1/2}$, however, then the question naturally arises as to whether or not they are optimal. 

In the present paper we consider the question of optimal rates of convergence for the quenched CLT of hitting times for one-dimensional RWRE. 
In particular, the results in this paper will slightly sharpen the upper bounds of the quantitative quenched CLT obtained in \cite{apQCLTrates} and will also provide lower bounds of the same polynomial order. 

\subsection{One-dimensional RWRE and previous results}

The model of one-dimensional (nearest neighbor) RWRE can be described as follows. 
An \emph{environment} is an element $\w = (\w_x)_{x \in \Z} \in (0,1)^\Z =: \Omega$. 
Given an environment $\w$ one can then define a random walk in the environment $\w$ to be a Markov chain $\{X_n\}_{n\geq 0}$ which when it is at site $x\in \Z$ steps to the right with probability $\w_x$ or to the left with probability $1-\w_x$. More precisely, for an environment $\w$ and a fixed starting point $z \in \Z$, we let $P_\w^z(\cdot)$ be the law of the Markov chain $\{X_n\}_{n\geq 0}$ given by 
\[
 P_\w^z(X_0 = z) = 1
\quad 
\text{and} 
\quad 
 P_\w^z\left(X_{n+1} = y \, | \, X_n = x \right) 
= \begin{cases}
   \w_x & \text{if } y=x+1 \\
   1-\w_x & \text{if } y=x-1 \\
   0 & \text{otherwise.}
  \end{cases}
\]
The distribution $P_\w^z$ is called the \emph{quenched} law of the RWRE and corresponding quenched expectations are denoted $E_\w^z[\cdot]$. 
To give some additional structure to the model, we assume that the environment $\w$ is chosen randomly according to some measure $P$ on the set of environments $\Omega$. 
One can then define the annealed (or averaged) law $\P^z$ of the RWRE by averaging the quenched law over the space of environments, 
\[
 \P^z(\cdot) = E\left[ P_\w^z(\cdot) \right]. 
\]
Annealed expectations are denoted $\E^z[\cdot]$. 
Since we will often be concerned with the random walk started at $z=0$, in this case for convenience of notation we will drop the superscript $z$ from the notation and write $P_\w, E_\w, \P$, and $\E$ instead. 

%Our main object of study in this paper will not be the position of the RWRE but instead the hitting times of the RWRE. 
%\[
% T_x = \inf\{k\geq 0:\, X_k = x \}, \qquad x \in \Z. 
%\]

Generally one may assume that the measure on environments is ergodic, in the present paper we will make the common assumption that the environment is i.i.d.
\begin{asm}\label{asm:iid}
 The distribution $P$ on environments is such that $\w = (\w_x )_{x\in \Z}$ is a sequence of i.i.d.\ random variables.  
\end{asm}
To describe the conditions under which a quenched CLT holds for the RWRE we need some additional notation. 
Let 
\[
 \rho_x = \frac{1-\w_x}{\w_x}, \quad x \in \Z. 
\]
The statistics of the random variable $\rho_0$ can be used to characterize many of the asymptotic behaviors of the RWRE. 
Our main assumption is the following. 
%The following assumption implies that the RWRE is transient to the right and satisfies a quenched CLT \cite{aRWRE}.
\begin{asm}\label{asm:CLT}
 $E[\log \rho_0] < 0$ and
 $E[\rho_0^2] < 1$. 
% $E[\rho_0^\k] = 1$ for some $\k > 2$. 
\end{asm}
We remark briefly on the importance of these assumptions. The assumption that $E[\log \rho_0] < 0$ implies that the RWRE is transient to the right \cite{sRWRE}, and for transient RWRE the assumption that $E[\rho_0^2]<1$ is enough to imply both quenched and annealed CLTs (though in the present paper we will only be interested in the quenched CLT). 
To state these results we first introduce notation for the hitting times of the random walk 
\[
 T_x = \inf\{k\geq 0:\, X_k = x \}, \qquad x \in \Z. 
\]
\begin{thm}[\cite{aRWRE,gQCLT,pThesis}]\label{th:QCLT}
 Let Assumptions \ref{asm:iid} and \ref{asm:CLT} hold. Then the following quenched CLTs hold for the hitting times and the position of the RWRE.  
\[
 \lim_{n\to\infty} P_\w\left( \frac{T_n - E_\w[T_n]}{\sqrt{\Var_\w(T_n)}} \leq x \right) = \int_{-\infty}^x \frac{1}{\sqrt{2\pi}} e^{-z^2/2} \, dz =: \Phi(x), \quad \forall x \in \R,
\]
and 
\[
 \lim_{n\to\infty} P_\w\left( \frac{X_n - n\vp + Z_n(\w)}{\s \vp^{3/2} \sqrt{n} } \leq x \right) = \Phi(x), \quad \forall x \in \R,
\]
where $\vp = 1/\E[T_1]$, $\s^2 = E[\Var_\w(T_1)]$, and $Z_n(\w) = \vp \left( E_\w[T_{\fl{n\vp}}] - \frac{\fl{n\vp}}{\vp} \right)$. 
\end{thm}

\begin{rem}
 A few remarks are in order regarding the centering and scaling terms in Theorem \ref{th:QCLT}. 
\begin{itemize}
 \item The constant $\vp$ is the limiting speed of the random walk; that is $\lim_{n\to\infty} X_n/n = \vp$, $\P$-a.s. The limiting speed is positive and equals $1/\E[T_1]$ in i.i.d.\ environments if and only if $E[\rho_0] < 1$ which in turn follows from the second part of Assumption \ref{asm:CLT}. 
 \item The scaling of the quenched CLT for the hitting times is different in \cite{aRWRE,gQCLT,pThesis} than as stated in Theorem \ref{th:QCLT}. Indeed, since it was shown in those papers that 
\begin{equation}\label{qVarlim}
 \lim_{n\to\infty} \frac{\Var_\w(T_n)}{n} = E[\Var_\w(T_1)] = \s^2, \quad P\text{-a.s.}
\end{equation}
 a quenched CLT for the hitting times also holds with scaling $\s \sqrt{n}$ as well. It was observed in \cite{apQCLTrates}, however, that using the quenched instead of deterministic scaling led to better quantitative error bounds for the quenched CLT. 
 \item While the quenched CLT is true with either a quenched or deterministic scaling, the environment dependent centering terms are necessary for a quenched CLT in i.i.d.\ environments. This can be seen for instance from the fact that $\frac{E_\w[T_n] - n/\vp}{\sqrt{n}}$ converges in distribution to a mean zero Gaussian. 
\end{itemize}
\end{rem}

%\begin{rem}
% We note that for one-dimensional RWRE in i.i.d.\ environments, both assumptions in Assumption \ref{asm:CLT} are necessary for one to have a quenched CLT. 
%If $E[\log\rho_0] = 0$ so that the walk is recurrent, then the scaling is $(\log n)^2$ and the limiting distribution is a non-Gaussian functional of Brownian motion \cite{sRecurrent}. 
%If $E[\log \rho_0] < 0$ but $E[\rho_0^2 ] > 1$ then there is no almost sure limiting distribution for the walk or the hitting times under the quenched measure \cite{pzSL1,p1LSL2}, 
%and if $E[\log \rho_0] < 0$ and $E[\rho_0^2 ] =1$ then one can get a weaker quenched CLT for the hitting times but where the convergence of the quenched distribution only holds in probability and the scaling is of the order $\sqrt{n \log n}$ \cite{dgWQL}. 
%\end{rem}

Quantitative versions of the quenched CLTs in Theorem \ref{th:QCLT} were studied recently in \cite{apQCLTrates}. The upper bounds on the error in the CLT was described in terms of the parameter
\begin{equation}\label{kdef}
 \k = \sup \{ \gamma >0: \, E[\rho_0^\gamma] < 1 \}. 
\end{equation}
Note that the second part of Assumption \ref{asm:CLT} is equivalent to the assumption $\k>2$. 
The parameter $\k>0$ has appeared in a number of previous results for one-dimensional RWRE. For instance, it has been shown to characterize the type of limiting distribution for transient RWRE \cite{kksStable,psWQLTn,psWQLXn,estzWQL,dgWQL} and also determine the subexponential rate of decay of certain large deviation probabilities \cite{dpzTE1D,gzQSETE,fgpMD}.
%Note that under rather mild additional assumptions the parameter $\k$ can also be defined by the equation $E[\rho_0^\k] = 1$, and in fact this is how it often appears in other places in the literature.
%Under rather mild additional assumptions we could also define $\k$ equivalently by $E[\rho_0^\k] = 1$. 
The following theorem from \cite{apQCLTrates} shows how the parameter $\k$ controls the polynomial rate of convergence of the quenched CLT for hitting times. 
\begin{thm}\label{th:BETn}
Let $F_{n,\w}(x) = P_\w\left(  \frac{T_n - E_\w[T_n]}{\sqrt{\Var_\w(T_n)}} \leq x \right)$ be the normalized quenched distribution of $T_n$. 
%Let $\Delta_{n,\w}^T = \sup_x \left| P_\w\left( \frac{T_n - E_\w[T_n]}{\sqrt{\Var_\w(T_n)}} \leq x \right) - \Phi(x) \right|$.
 \begin{enumerate}
  \item If $\kappa > 3$, then there exists a constant $C\in(0,\infty)$ such that 
\[
 \limsup_{n\ra\infty}  \sqrt{n} \left\| F_{n,\w} - \Phi \right\|_\infty \leq C, \qquad P\text{-a.s.}
\]
 \item If $\k \in (2,3]$, then for any $\e>0$
\[
  \lim_{n\ra\infty} n^{\frac{3}{2}-\frac{3}{\k}-\e} \left\| F_{n,\w} - \Phi \right\|_\infty  = 0, \qquad P\text{-a.s.}
\]
 \end{enumerate}
\end{thm}
%It is easy to see that since $T_n$ is integer valued and and the quenched scaling is almost surely asymptotic to $\s \sqrt{n}$ by \eqref{qVarlim} that for any $\k>2$ one always has 
%the following lower bound on the error in the quenched CLT for the hitting times:  
%\[
% \liminf_{n\to\infty} \sqrt{n} \left\| F_{n,\w} - \Phi \right\|_\infty > 0. 
%\]
%Thus the polynomial rate of decay in Theorem \ref{th:BETn} is optimal when $\k>3$. 
%It is not as clear, however, if the polynomial rates of decay in Theorem \ref{th:BETn} are optimal when $\k \in (2,3]$ and this is the focus of the present paper. 
It is not hard to see that the polynomial rate of convergence in Theorem \ref{th:BETn} is optimal in the case $\k>3$. 
Indeed, since the distribution function $F_{n,\w}$ is constant on intervals of length $1/\sqrt{\Var_\w(T_n)}$, by considering the interval containing the origin for which $F_{n,\w}$ is constant we can see that 
\[
 \left\|F_{n,\w} - \Phi \right\|_\infty 
%\geq \sup_{x \in \left[ -\frac{E_\w[T_n]- \fl{E_\w[T_n]}}{\sqrt{\Var_\w(T_n)}} ,  \frac{1-(E_\w[T_n]- \fl{E_\w[T_n]})}{\sqrt{\Var_\w(T_n)}} \right) } \left| \overline{F}_{n,\w}(x) - \Phi(x) \right| 
\geq \frac{1}{2} \left( \Phi\left( \frac{1}{\sqrt{\Var_\w(T_n)}} \right) - \frac{1}{2} \right). 
\]
%Since $\k>2$ implies that $\Var_\w(T_n)/n \to \s^2 := E[ \Var_\w(T_1)]$, $P$-a.s., {\color{red}(probably should reference this fact somewhere)} 
Using \eqref{qVarlim} it follows that $\liminf_{n\to\infty} \sqrt{n} \left\| F_{n,\w} - \Phi \right\|_\infty \geq \frac{1}{2\sqrt{2\pi}\s}$, $P$-a.s.
Note that this lower bound on the rate of convergence holds for any $\k>2$, but only in the case $\k>3$ does the polynomial rate match the one in the upper bound in Theorem \ref{th:BETn}. 
Our focus on the present paper, therefore, is in identifying the optimal polynomial rate of convergence when $\k \in (2,3]$.

\subsection{Main results}

To obtain the precise polynomial rate of decay for the error in the quenched CLT when $\k \in (2,3]$ we will need the following additional technical assumptions on the distribution $P$ on environments. 
%In the remainder of the paper we will need the following additional technical assumptions on the distribution on environments. 
\begin{asm}\label{asm:tech}
 The distribution of $\log \rho_0$ is non-lattice (that is, $\P( \log \rho_0 \in a \Z) < 1$ for all $a>0$) and $E[\rho_0^\k \log \rho_0] < \infty$, where $\k$ is defined as in \eqref{kdef}. 
\end{asm}
\begin{rem}
 We note that assumptions are quite weak and have been assumed in a number of previous papers such as \cite{kksStable,dgWQL,estzWQL,fgpMD,gsMVSS,psWQLTn,pzSL1}. 
 While these assumptions are certainly necessary in some of these previous papers, we suspect that our main result is in fact true without this additional assumption, but our proof is simplified by making this additional assumption so that we can freely use some precise tail asymptotics which were obtained previously using this assumption. 
\end{rem}

Our first main result shows that if $\k \in (2,3)$ then the error in the quenched CLT is of the order $n^{-\frac{3}{2} + \frac{3}{\k}}$ but with an environment-dependent multiplicative term that oscillates between 0 and $\infty$. 
\begin{thm}\label{th:ExactRates}
 Let Assumptions \ref{asm:iid}, \ref{asm:CLT} and \ref{asm:tech} hold, and let $F_{n,\w}$ be as in Theorem \ref{th:BETn}
and let $\k \in (2,3)$. Then,
\begin{equation}\label{LI}
 \liminf_{n\to\infty} n^{\frac{3}{2}-\frac{3}{\k}} \left\| F_{n,\w} - \Phi \right\|_\infty  = 0, \qquad P\text{-a.s.}, 
\end{equation}
and 
\begin{equation}\label{LS}
 \limsup_{n\to\infty} n^{\frac{3}{2}-\frac{3}{\k}} \left\| F_{n,\w} - \Phi \right\|_\infty  = \infty, \qquad P\text{-a.s.}.
\end{equation}
\end{thm}
Our second main result shows that the asymptotics are a bit more delicate when $\k=3$ since the optimal rate of convergence involves a logarithmic term as well. 
\begin{thm}\label{th:k3exact}
 Let Assumptions \ref{asm:iid}, \ref{asm:CLT} and \ref{asm:tech} hold with $\k=3$, and let $F_{n,\w}$ be as in Theorem \ref{th:BETn}. 
Then, there exist constants $0<\ell<u<\infty$ such that 
\[
 \lim_{n\to\infty} P\left( \ell \leq \frac{\sqrt{n}}{\log n} \| F_{n,\w} - \Phi \|_\infty \leq u \right) = 1. 
\]
\end{thm}

The main idea behind proof of Theorem \ref{th:ExactRates} is the same as what was used in the original proof of the quenched CLT in \cite{aRWRE}: under the quenched measure the hitting time $T_n$ is the sum of independent but not identically distributed random variables. That is, $T_n = \sum_{i=1}^n \tau_i$ where $\tau_i = T_i-T_{i-1}$ and the random variables $\tau_i$ are independent under the quenched measure. 
The quenched CLT for $T_n$ can then be obtained by checking the conditions of the Lindberg-Feller CLT. 
Similarly, to obtain quantitative bounds on the quenched CLT we can apply previous results for quantitative CLT bounds for sums of independent random variables. The resulting bounds are expressed in terms of sums of quenched moments of hitting times and the main work in proving Theorem \ref{th:ExactRates} is in providing good control on these sums of quenched moments. 

We note that the control of the upper bound in \eqref{LI} is obtained 
by using the Berry-Esseen upper bounds for sums of independent random variables. 
This gives an upper bound in terms of sums of quenched third moments of hitting times. This is the same approach that was used in \cite{apQCLTrates} but we improve on the upper bound $\k \in (2,3]$ in Theorem \ref{th:QCLT} by more carefully handling the sums of the quenched third moments. 

While it is relatively easy to find results which give some quantitative upper bounds on approximation of sums of independent random variables by a Normal distribution, it seems that results giving lower bounds on the normal approximation are less well known. 
To obtain lower bounds for the quenched CLT when $\k \in (2,3]$ we follow an approach from \cite{hbReversing} which uses Stein's method. The result in \cite{hbReversing}, while optimal in some respects does not lead to the right lower bounds in our application. However, we are able to make a simple modification of the proof in that paper to obtain a new lower bound which leads to the optimal lower bound in \eqref{LS}. 

\subsection{Discussion and further questions}

As noted above, Theorem \ref{th:ExactRates} implies that when $\k \in (2,3)$ the maximal error in the quenched CLT is of the order $n^{-\frac{3}{2}+\frac{3}{\k}}$ but with a multiplicative constant depending on $\w$ and $n$ which oscillates between 0 and $\infty$. 
In fact, we conjecture that the following holds. 
\begin{conj}\label{con:weaklimit}
 If $\k \in (2,3)$, then $n^{\frac{3}{2}-\frac{3}{\k}} \left\| F_{n,\w} - \Phi \right\|_\infty$ converges in distribution to a non-degenerate random variable on $(0,\infty)$. 
\end{conj}
%Both the upper and lower bounds we obtain in the course of the proof of Theorem \ref{th:ExactRates} are given in terms of sums of quenched third moments of hitting times. We suspect that a more careful analysis might be able to prove a limiting distribution for these sums of quenched third moments. However, since the upper and lower bounds also have different multiplicative constants even such a limiting distribution for the sums of quenched third moments would not be enough to prove Conjecture \ref{con:weaklimit}.  
We suspect that using the upper and lower bounds on $n^{\frac{3}{2}-\frac{3}{\k}} \left\| F_{n,\w} - \Phi \right\|_\infty$ given in the proof of Theorem \ref{th:ExactRates} both converge in distribution to non-degenerate random variables on $(0,\infty)$. Proving this would require some technical work, but should be doable. However, we do not think that the limiting distributions of the upper and lower bounds would be the same limit, and so different techniques would be needed to prove Conjecture \ref{con:weaklimit}.

The quenched CLT for the position of the RWRE in Theorem \ref{th:QCLT} can be obtained from quenched CLT for the hitting times. 
This can be done by first showing that $\lim_{n\to\infty} \frac{X_n^* - X_n}{(\log n)^2}= 0$, $\P$-a.s., where $X_n^* = \max_{k\leq n} X_k$ is the running maximum of the walk, and then by noting that
$P_\w( X_n^* < k ) = P_\w(T_k > n)$ to transfer limiting distribution statements for the hitting times to limiting distributions for $X_n^*$ (and thus also for $X_n$). 
This approach was also used in \cite{apQCLTrates} to transfer the results in Theorem \ref{th:QCLT} to quantitative  upper bounds on the quenched CLT for $X_n$. 
In particular, it was shown that if $\k>2$ and  $G_{n,\w}(x) = P_\w\left( \frac{X_n - n\vp + Z_n(\w)}{\s \vp^{3/2} \sqrt{n}} \leq x \right)$ then for any $\e>0$ we have 
\begin{equation}\label{Gnrates}
 \lim_{n\to\infty} n^{\frac{1}{4} - \frac{1}{2\k}-\e} \| G_{n,\w} - \Phi \|_\infty = 0, \quad P\text{-a.s.} 
\end{equation}
It seems, however, not to be possible to use the results of this paper to obtain quantitative lower bounds for the quenched CLT for $X_n$.
This is, in part, because the lower bounds for the error in normal approximation we use do not identify the location where the difference in the distribution functions are maximized. 
Thus, it remains open to check if the rates in \eqref{Gnrates} are (nearly) optimal or if they can be improved by different methods.

\section{Notation and preparatory lemmas}

This section introduces notations and state preliminary formulas and lemmas that will help to prove our main theorem.  Recall that for a given environment $\w = (\w_x)_{x\in\Z}$, we have defined $\rho_x = \frac{1-\w_x}{\w_x}$. Then, for any integers $i\leq j$ we define
	$$\Pi_{i,j}:=\prod_{k=i}^j\rho_k,\quad W_{i,j}:=\sum_{k=i}^j\Pi_{k,j},\quad R_{i,j}:=\sum_{k=i}^j\Pi_{i,k}
	$$
	$$W_j:=\sum_{k\leq j}\Pi_{k,j},\quad R_i:=\sum_{k=i}^\infty\Pi_{i,k}.
	$$
	These notations 
%are introduced by Peterson in \cite{pThesis} 
were introduced in \cite{pzSL1}
to simplify certain lengthy exact formulas for quenced probabilities or expectations.
For instance, the quenched expectation and variance of the time to cross from $i-1$ to $i$ are given by  
	\begin{equation}\label{QE}
	E_\w[\t_i]=1+2W_i.
	\end{equation}
and 
	\begin{equation}  \label{QV}
	\Var_\w(\t_i)=4(W_{i-1}+W_{i-1}^2)+8\sum_{j<i-1}\Pi_{j+1,i}(W_j+W_j^2).
	\end{equation}
\begin{rem}
 In \eqref{QE} and throughout the paper we will use $\tau_i$ to denote a random variable with the distribution of the time it takes a walk started at $i-1$ to reach $i$ for the first time. 
When $i\geq 1$ we can (and do) identify $\tau_i$ with $T_i-T_{i-1}$, but of course this does not work for $i\leq 0$. 
\end{rem}

For the proof of our main results we will also need explicit formulas for the quenched third moments of hitting times. For this, we will need the following lemma whose proof is given in Appendix \ref{app:3m}.  
	\begin{lem}\label{lem:thirdmoment}
	For $P$-a.e.\ environment $\w$ and every $i\in \Z$ we have 
	\[E_\w[\t_i^3]=\frac{1}{\w_{i-1}}+6E_\w[\t_i]\Var_\w(\t_i)+\r_{i-1} E_\w[\t_{i-1}^3].\]	
	where $E_\w[\t_i]$ and $\Var_\w(\t_i)$ are explicitly defined in \eqref{QE} and \eqref{QV} respectively. 
	\end{lem}

For higher quenched moments of crossing times we will not need explicit formulas like in \eqref{QE}, \eqref{QV} or Lemma \ref{lem:thirdmoment} (though these can be obtainded), but we will need some control on the distribution of these quenced higher moments. For this we recall the following result from \cite{apQCLTrates} which will be used frequently in the proof of our main results. 
%We also state a following preliminary lemma of the finite boundness condition on quenched moment of $T_1$.  The lemma is frequently used in the proof to control the moments of crossing time in the proof of our main theorem.
\begin{lem}[\cite{apQCLTrates}, Lemma 2.1]\label{lem:qtmom}
	If $E[\log \rho_0] < 0$ and $\k>0$, then for any integer $m\geq 1$, 
%$E[E_\w[T_1^m]^p] < \infty$ for all $p < \frac{\k}{m}$. 
\[
 E\left[E_\w[T_1^m]^p \right] < \infty, \qquad \forall p < \frac{\k}{m}.
\]
	%Moreover, this implies that $\E[T_1^p]<\infty$ for all $p < \k$. 
\end{lem}
\begin{rem}\label{rem:annmom}
	Note that Lemma \ref{lem:qtmom} implies that $\E[T_1^p]<\infty$ for all $p < \k$. 
	Indeed, let $p<\k$ and choose any integer $m\geq p$. Since $p/m\leq 1$ it follows from Jensen's inequality that $E_\w[T_1^p] \leq E_\w[T_1^m]^{p/m}$. Thus, Lemma \ref{lem:qtmom} implies that
	$\E[T_1^p] \leq E[E_\w[T_1^m]^{p/m}] < \infty$.
	When $\k \geq 1$ it is easy to show that the reverse is also true; that is, $\E[T_1^p]=\infty$ for all $p \geq \k$. 
	Indeed, if $p\geq 1$ then Jensen's inequality implies that $\E[T_1^p] \geq E[E_\w[T_1]^p]$ and since it follows from \cite{kRDE} that $P(E_\w[T_1] > x) \sim C x^{-\k}$ we have that $E[E_\w[T_1]^p] = \infty$ for all $p\geq \k$. 
\end{rem}
	   
	Another key ingredient to prove our main theorem is a method developed by Sinai of the ``potential'' of an environment in \cite{sRecurrent}.  The method was initially used to study the limiting distributions of recurrent RWRE but has also shown to be useful for transient RWRE \cite{pzSL1},\cite{fgpMD},\cite{psWQLTn},\cite{estzWQL}, \cite{apOscill}. For a fixed environment $\omega$, let the potential $V(x)$ be the function
	
	\[V(x) =
	\begin{cases}\
	\sum_{i=0}^{x-1}\log\rho_i & \quad \text{if $x\geq1$}\\
	0 & \quad \text{if $x=0$}\\
	-\sum_{i=x}^{-1}\log\rho_i & \quad \text{if $x\leq-1$}.
	\end{cases}\]
	
	This method turns a series of transition probabilities into a landscape image, $V(x)$, on $\Z$.     Using potential $V(x)$, we can split the environment into blocks by ``ladder location'', $\{\nu_i, i\in\mathbb{Z}\},$ given by
	\begin{equation}
	\nu_{0}=\sup\{y\leq 0:V(y)<V(k),\forall k<y\},\label{nu0}
	\end{equation}
	and for $i\geq 1$,
	$$\nu_i=\inf\{x>\nu_{i-1}:V(x)< V(\nu_{i-1})\}, \quad\text{and}\quad \nu_{-i}=\sup\{y<\nu_{-i+1}:V(y)<V(k),\forall k<y\}.$$
%	Equivalently,
%	$$\nu_0=\sup\{y\leq 0:\Pi_{k,y-1}<1,\forall k<y\}$$
%	and, for $i\geq 1$,
%	$$\nu_i=\inf\{x>\nu_{i-1}:\Pi_{\nu_{i-1},x-1}<1\},\quad \text{and}\quad\nu_{-i}=\sup\{y<\nu_{-i+1}:\Pi_{k,y-1}<1,\forall k<y\}.$$
	Each ladder point is roughly characterized as a bottom of valley with a up-hill on the left in the landscape image, 
and the ladder location $\nu_0$ identifies the first such valley to the left of the origin. 
Note that since the environment $\w$ is i.i.d., the potential $V$ is a sum of i.i.d.\ random variables, and thus the ladder locations can be viewed as points of a stationay renewal process. 
Thus, if we let $\mathfrak{B}_i = \{\omega_x: \, x \in [\nu_i,\nu_{i+1}) \}$ denote the ``blocks'' of the environment between ladder locations then it follows that the blocks $\{\mathfrak{B}_i\}_{i\in \Z}$ are independent but only identically distributed for $i\neq 0$ (this is an instance of the inspection paradox). 
 For this reason, it is sometimes convenient to work with a related measure on environments $Q$ introduced in \cite{pzSL1} and given by
	\[
	Q(\cdot) = P( \cdot | \nu_0 = 0 ).
	\]
	%The measure $Q$ was also used previously in \cite{PZ09,Pet09,PS13}.
	Under measure $Q$, the sequence $\{\omega_x\}_{x\in\mathbb{Z}}$ is no longer i.i.d.\ but this distribution has the convenient property that the environment is stationary under shifts of the ladder points of the environment.  More precisely, if $\theta$ is the natural left-shift operator on environments given by $(\theta \omega)_x=\omega_{x+1}$, then for any $k\in \mathbb{Z}$ the environments $\omega$ and $\theta^{\nu_k}\omega$ have the same distribution under $Q$.  Moreover, under the measure $Q$ the blocks between adjacent ladder points $\mathfrak{B}_i$ are i.i.d.\ for all $i \in \mathbb{Z}$ with each having the same distribution as $\mathfrak{B}_1$ under the original measure $P$ on environments. 

We conclude this section with some tail asymptotics that will be used throughout the paper. We will often be concerned with the time it takes the random walk to cross from one ladder location to the next, and it turns out that a good rough measure of the time it takes for the random walk to do this is given in terms of the height of the potential between these ladder locations. 
For this reason 
%Then, each block between the consecutive ladder points contains a hill with its size measured by the height of potential.  
let us define the exponential height of the potential between ladder points by
\begin{equation}\label{Midef}
 M_i:=\max\{e^{V(j)-V(\nu_{i})}:\nu_{i-1}<j\leq \nu_i\}=\max\left\{\Pi_{\nu_{i-1},j}:\nu_{i-1}\leq j<\nu_i\right\}, \qquad i \in \mathbb{Z}.
\end{equation}
%	$$M_i:=\max\{e^{V(j)-V(\nu_{i})}:\nu_{i-1}<j\leq \nu_i\}=\max\left\{\Pi_{\nu_{i-1},j}:\nu_{i-1}\leq j<\nu_i\right\}, \qquad i \in \mathbb{Z}.$$
As the exponential heights depend only on the environment between ladder locations, it follows that the $M_i$ are i.i.d.\ for $i\neq 0$ under the measure $P$ and i.i.d.\ for all $i$ under the measure $Q$. Moreover, the distribution of $M_i$ for $i\neq 0$ is the same under both $P$ and $Q$, and it follows from a result of Iglehart \cite{Igl72} that $M_1$ has heavy tails; that is, 
there exists a constant $C_0>0$ such that
	\begin{equation}
	Q(M_1>x)\sim C_0 x^{-\k},\quad\text{ as }x\to\infty.\label{M}
	\end{equation}
	It turns out that the exponential height $M_i$ has a crucial role to determine the size of quenched expectation and variance of a time to cross from $\nu_{i-1}$ to ${\nu_i}$, denoted by $E_\w^{\nu_{i-1}}[T_{\nu_i}]$ and $\Var_\w(T_{\nu_i}-T_{\nu_{i-1}})$ respectively.   Roughly, the size of $M_i$ is comparable to $E_\w^{\nu_{i-1}}[T_{\nu_i}]$, and $M_i^2$ is comparable to $\Var_\w(T_{\nu_i}-T_{\nu_{i-1}})$. Indeed, it was shown in \cite[Theorem 1.4 and Theorem 5.1]{pzSL1} that
	\begin{equation}\label{Qbtails}
	Q(E_\w[T_{\nu_1}]>x)\sim K_\infty x^{-\k},\quad\text{and}\quad Q(\Var_\w(T_{\nu_1})>x)\sim K_\infty x^{-\k/2},\quad\forall x\geq0,
	\end{equation}
	for some $K_\infty>0$.  

\section{Upper bound on rates of convergence}

As in \cite{apQCLTrates} the starting point for our upper bound are the Berry-Esseen bounds for sums of independent random variables. 
\begin{thm}[Theorem V.3.6 in \cite{petrovRW}]\label{BEind}
 Let $\xi_1,\xi_2,...,\xi_n$ be independent random variables which have zero mean and finite 
 third moments. 
%variances $E[\xi_i^2]=\sigma_i^2,\,1\leq i\leq n$, and satisfy $\sum_{i=1}^n\sigma_i^2=1$.  
 Let $Z_n=\sum_{i=1}^n\xi_i$.  Then there exists an absolute constant $C_1 \in (0,\infty)$ such that 
\[
 \sup_x \left| P\left( \frac{Z_n}{\Var(Z_n)} \leq x \right) - \Phi(x) \right| \leq \frac{C_1 \sum_{i=1}^n E[|\xi_i|^3]}{\Var(Z_n)^{3/2}}.   
\]
\end{thm}
Applying Theorem \ref{BEind} to the random variables $\{\tau_i - E_\w[\tau_i]\}_{i\leq n}$ (which are independent under the quenched measure) and then noting that
\begin{align*}
  E_\w[|\tau_k - E_\w[\tau_k]|^3]
&\leq E_\w\left[ (\tau_k + E_\w[\tau_k])(\tau_k - E_\w[\tau_k])^2 \right] \\
%&= E_\w\left[ (\tau_k + E_\w[\tau_k])(\tau_k^2 - 2 \tau_k E_\w[\tau_k] + E_\w[\tau_k]^2 ) \right] \\
%&= E_\w\left[ \tau_k^3 -2 \tau_k^2 E_\w[\tau_k] + \tau_k (E_\w[\tau_k])^2 + \tau_k^2 E_\w[\tau_k]- 2 \tau_k (E_\w[\tau_k])^2 + (E_\w[\tau_k])^3 \right] \\
&= E_\w\left[ \tau_k^3 - \tau_k^2 E_\w[\tau_k] - \tau_k (E_\w[\tau_k])^2 + (E_\w[\tau_k])^3 \right] \\
%&= E_\w[\tau_k^3] - E_\w[\tau_k^2]E_\w[\tau_k] \\
&\leq E_\w[\tau_k^3] 
\end{align*}
 we obtain that 
\begin{equation}\label{bound1}
 \sup_x \left| F_{n,\w}(x) - \Phi(x) \right| 
%\leq \frac{C_1}{\Var_\w(T_n)^{3/2}} \sum_{k=1}^n E_\w\left[ \left| \tau_k - E_\w[\tau_k] \right|^3 \right].
\leq \frac{C_1}{\Var_\w(T_n)^{3/2}} \sum_{k=1}^n E_\w\left[\tau_k^3 \right].
\end{equation}
Since $\Var_\w(T_n)^3 \sim \s^3 n^{3/2}$ by \eqref{qVarlim}, 
the liminf \eqref{LI} in Theorem \ref{th:ExactRates} will then follow immediately from the following proposition. 
\begin{prop}\label{prop:uppermain}
If $\k \in (2,3)$, then, 
\[
 \liminf_{n\to\infty} \frac{1}{n^{3/\k}} \sum_{k=1}^n E_\w[ \tau_k^3 ] = 0, \quad P\text{-a.s.}
\]
\end{prop}
\begin{proof}
To prove Proposition \ref{prop:uppermain}, we first make two observations. First, since we are taking a liminf we can take the limit along any subsequence of sites which is convenient and so we will use the sequence of ladder locations $\{\nu_i \}_{i\geq 1}$. Since $\nu_n/n \to E[\nu_1]<\infty$, $P$-a.s.\ as $n\to\infty$, it is therefore enough to show that 
\begin{equation}\label{umP}
 \liminf_{n\to\infty} \frac{1}{n^{3/\k}} \sum_{k=1}^{\nu_n} E_\w[ \tau_k^3 ] = 0, \quad P\text{-a.s.}
\end{equation}
In fact, this statement will follow if we can prove the same limit under the measure $Q$ on environments instead. 
That is, we will show the following. 
\begin{prop}\label{prop:uppermain2}
	If $\k \in (2,3)$,
	\[\liminf_{n\to\infty}\frac{1}{n^{3/\kappa}}\sum_{i=0}^{\nu_n} E_\omega[\t_i^3]=0,\quad Q-a.s.\]
\end{prop}
We will defer the proof of Proposition \ref{prop:uppermain2} for the moment to show how \eqref{umP} follows from this. 
First of all, let $\tilde{P}$ be the measure on environments which is obtained by sampling an environment according to measure $P$ and then shifting the ladder point $\nu_0 \leq 0$ to the origin. That is,  $\tilde{P}(\w\in\cdot)=P(\theta^{\nu_0}\w\in\cdot)$.
Since $\tilde{P}$ is absolutely continuous with respect to $Q$ (see \cite[Lemma 4.2]{apOscill})
it follows from Proposition \ref{prop:uppermain2} that 
\[
 \liminf_{n\to\infty}\frac{1}{n^{3/\kappa}}\sum_{i=1}^{\nu_n} E_\omega[\t_i^3]=0,\quad \tilde{P}-a.s.,
\]
or equivalently 
\[
 \liminf_{n\to\infty}\frac{1}{n^{3/\kappa}}\sum_{i=\nu_0 + 1}^{\nu_n} E_\omega[\t_i^3]=0,\quad P-a.s.
\]
Since $\nu_0 \leq 0$ this implies \eqref{umP}. 
Pending the proof of Proposition \ref{prop:uppermain2} this completes the proof of Proposition \ref{prop:uppermain}. 
\end{proof}

The remainder of this section will be devoted toward proving Proposition \ref{prop:uppermain2}. 
Our first step is the following lemma which gives a convenient upper bound for sums of quenched third moments in terms of lower moments of crossing times of a larger interval. 
\begin{lem}\label{lem:Et3sum}
	\begin{equation}\label{sum:upper}
	\sum_{i=1}^n E_\w[\t_i^3]\leq E_\w[T_n]+6E_\w[T_n]\Var_\w(T_n)+R_{0,n-1}E_\w[\t_0^3]
	\end{equation}
\end{lem}

%Then equation \eqref{thirdmoment} leads a rough upper bound of the first $n$ sums of the third moment of $\t_i$ by the combination of $E_\w[T_n]$ and $\Var_\w(T_n)$ as the following lemma.

\begin{proof}
We begin with the following recursive equation for the third moment of the crossing time $\tau_i$ in Lemma \ref{lem:thirdmoment} given by  
\begin{equation}\label{third:decomp}
E_\w[\t_i^3]=\frac{1}{\w_{i-1}}+6E_\w[\t_i]\Var_\w(\t_i)+\r_{i-1} E_\w[\t_{i-1}^3]
\end{equation}
Iterating \eqref{third:decomp} a total of $i$ times we get
\begin{equation}\label{thirdmoment}
E_\w[\t_i^3]=1+2W_{1,i-1} + \Pi_{0,i-1}+6\sum_{0<j\leq i}\Pi_{j,i-1}E_\w[\t_j]\Var_\w(\t_j)+\Pi_{0,i-1} E_\w[\t_0^3]
\end{equation}
%	Using the equation in \eqref{thirdmoment}, the sums of the third moments of the crossing time $\{\t_i\}_{1\leq i\leq n}$ is
Summing this over $i$ from $1$ to $n$ we then get 
	\begin{align}
	\sum_{i=1}^n E_\w[\t_i^3]&= \sum_{i=1}^{n}(1+2W_{1,i-1} + \Pi_{0,i-1})+6\sum_{i=1}^n\sum_{0<j\leq i}\Pi_{j,i-1}E_\w[\t_j]\Var_\w(\t_j)+\sum_{i=1}^n\Pi_{0,i-1} E_\w[\t_0^3] \nonumber \\
	&\leq \sum_{i=1}^{n}(1+2W_{i-1})+6\sum_{i=1}^n\sum_{0<j\leq i}\Pi_{j,i-1}E_\w[\t_j]\Var_\w(\t_j)+R_{0,n-1} E_\w[\t_0^3] \nonumber \\
	&= E_\w[T_n]+6\sum_{i=1}^n\sum_{0<j\leq i}\Pi_{j,i-1}E_\w[\t_j]\Var_\w(\t_j)+R_{0,n-1} E_\w[\t_0^3], \label{Et3d3}
	\end{align}
where in the last line we used that $E_\w[\tau_i] = 1 + 2W_{i-1}$.
	It remains to bound the middle term in \eqref{Et3d3} by $6 E_\w[T_n]\Var_\w(T_n)$. 
	To this end, note for $j\leq i$ that 
	\[
	\Pi_{j,i-1}E_\w[\t_j] = \Pi_{j,i-1} (1+2W_{j-1}) 
	\leq 2 W_{i-1} < E_\w[\tau_i], 
	\]
	and thus we may bound the middle term in \eqref{Et3d3} by 
	\[
	6 \sum_{i=1}^n\sum_{0<j\leq i} E_\w[\tau_i] \Var_\w(\t_j)
	= 6 \sum_{j=1}^n  \left(\sum_{i=j}^n E_\w[\tau_i] \right)\Var_\w(\t_j)
	\leq 6 E_\w[T_n] \Var_\w(T_n). 
	\]
\end{proof}
Lemma \ref{lem:Et3sum} is too rough an upper bound to be able to obtain Proposition \ref{prop:uppermain} directly
since the second term in the upper bound grows quadratically in $n$ and $n^2 \gg n^{3/\k}$. 
%we have that $E_\w[T_n] \Var_\w(T_n) \sim \frac{\s^2}{\vp} n^2$ as $n\to\infty$.
However, Lemma \ref{lem:Et3sum} does give a good upper bound when applied to intervals between ladder locations. 
To this end, we introduce the notation 
\[
 \mu_{i,\w}=E_{\w}^{\nu_{i-1}}[T_{\nu_i}]\quad \text{ and }\quad 
\sigma^2_{i,\w}=\Var_{\w}(T_{\nu_{i}}-T_{\nu_{i-1}})
\]
for the quenched mean and variance of the time to cross from $\nu_{i-1}$ to $\nu_i$. 
By applying Lemma \ref{lem:Et3sum} to a sums of the third moment of crossing times between consecutive ladder locations we obtain 
\begin{equation}
\sum_{j=\nu_{i-1}+1}^{\nu_i}E_\w[\t_j^3]\leq \mu_{i,\w}+6\mu_{i,\w}\sigma^2_{i,\w}+R_{\nu_{i-1},\nu_i-1}E_\w[\t_{\nu_{i-1}}^3], \nonumber
\end{equation}
and then summing this for $i\leq n$ gives
\begin{align}
\sum_{i=1}^{\nu_n}E_\w[\t_i^3]&\leq \sum_{i=1}^{n} \mu_{i,\w}+6\sum_{i=1}^{n}\mu_{i,\w}\sigma^2_{i,\w}+\sum_{i=1}^{n} R_{\nu_{i-1},\nu_i-1}E_\w[\t_{\nu_{i-1}}^3]\nonumber\\
&=\sum_{i=1}^{n} \mu_{i,\w}+6\sum_{i=1}^{n} \mu_{i,\w}(\sigma^2_{i,\w}-\mu_{i,\w}^2)+6\sum_{i=1}^{n} \mu_{i,\w}^3+\sum_{i=1}^{n} R_{\nu_{i-1},\nu_i-1}E_\w[\t_{\nu_{i-1}}^3].\label{decomp:thirdmoment}
\end{align}
%NOTE: SOME OF THE FOLLOWING ANALYSIS IS DONE USING RESULTS FROM MY THESIS WHICH AREN'T FOR THE ORIGINAL RWRE BUT FOR THE RWRE WITH ADDED REFLECTIONS THAT ARE $-(\log n)^2$ TO THE LEFT OF THE CURRENT LOCATION. WE NEED TO VERIFY THAT WE CAN REDUCE THINGS TO THIS CASE.
The third term in \eqref{decomp:thirdmoment} will be the main term since the following Lemma implies that the other three terms grow only linearly in $n$. 
\begin{lem}\label{lem:nonmain}
If $\k>2$ then there exists a constant $A<\infty$ such that 
 \[
  \lim_{n\to\infty} \frac{1}{n} \left\{ \sum_{i=1}^{n} \mu_{i,\w}+6\sum_{i=1}^{n} \mu_{i,\w}(\sigma^2_{i,\w}-\mu_{i,\w}^2)+\sum_{i=1}^{n} R_{\nu_{i-1},\nu_i-1}E_\w[\t_{\nu_{i-1}}^3] \right\}
= A, \quad Q\text{-a.s.}
 \]
\end{lem}
\begin{proof}
 Since the terms in each of the three sums are ergodic sequences under the measure $Q$, we need only to show that they all have finite mean under $Q$. For the first term we have $E_Q[\mu_{1,\w}] < \infty$ whenever $\k>1$, so we need only show that
% $E_Q[ \mu_{i,\w}(\sigma^2_{i,\w}-\mu_{i,\w}^2)] < \infty $ and $E_Q[ R_{\nu_{i-1},\nu_i-1}E_\w[\t_{\nu_{i-1}}^3] < \infty$. 
\begin{equation}\label{finmean}
 E_Q[ \mu_{1,\w}(\sigma^2_{1,\w}-\mu_{1,\w}^2)] < \infty
 \quad \text{and}\quad
 E_Q[ R_{0,\nu_1-1}E_\w[\t_{0}^3]] < \infty.
\end{equation}
For the first claim in \eqref{finmean}, it follows from equation (49) and Corollary 5.4 in \cite{pzSL1} that $Q( \sigma^2_{1,\w}-\mu_{1,\w}^2 > x) = o(x^{-\k+\e})$ for any $\e>0$, 
which implies that $E_Q\left[ (\sigma^2_{1,\w}-\mu_{1,\w}^2 )^2 \right] < \infty$ whenever $\k>2$. Since $E_Q[\mu_{1,\w}^2]<\infty$ also whenever $\k>2$, the first claim in \eqref{finmean} follows from the  Cauchy-Schwartz inequality. 

For the second claim in \eqref{finmean}, first note that since $E_\w[\tau_0^3] \in \s(\w_x, \, x\leq -1)$ and $R_{0,\nu_1-1} \in \s(\w_x, \, x\geq 0)$ we have that
\[
E_Q[ E_\w[\tau_0^3]R_{0,\nu_1-1}] = E_Q[ E_\w[\tau_0^3] ] E[ R_{0,\nu_1-1}] \leq  E_Q[ E_\w[\tau_0^3] ] E[ R_0 ]
= E_Q[ E_\w[\tau_0^3] ] \frac{ E[\rho_0]}{1-E[\rho_0]}, 
\]
where the last equality holds when $E[\rho_0] < 1$ (i.e., $\k>1$). 
Thus, we need to show that 
% $E_Q[ E_\w[\tau_0^3] ] < \infty$. 
\begin{equation}\label{tau03m}
 E_Q[ E_\w[\tau_0^3] ] < \infty.
\end{equation}
To prove \eqref{tau03m}, first note that it follows from iterating the recursive formula for the third moment in Lemma \ref{lem:thirdmoment} that\footnote{Clearly one can iterate the recursive formula in Lemma \ref{lem:thirdmoment} finitely many times to obtain partial sums of the right side of \eqref{recursive} plus an error term. It is not too difficult then to show that the error term vanishes as the number of recursions increases. In fact, it was shown in \cite[pages 1393-1394]{apQCLTrates} that this works for finding a formula for $E_\w[\tau_i^m]$ for any integer $m\geq 1$.}
	\begin{equation}\label{recursive} 
	E_\w[\t_{0}^3]=E_\w[\t_{0}]+6\sum_{k=1}^\infty \Pi_{-k+1,-1}E_\w[\t_{-k+1}]\Var(\t_{-k+1}).
	\end{equation}
	We split the sum of the second term by the groups of sums between ladder locations $\{\nu_{j}\}_{j\leq 0}$ to the left of the origin. 
	\begin{align*}
	\sum_{k=1}^\infty \Pi_{-k+1,-1}E_\w[\t_{-k+1}]\Var(\t_{-k+1})&=\sum_{j\leq 0} \sum_{x=\nu_{j-1}+1}^{\nu_{j}}\Pi_{x,-1}E_\w[\t_x]\Var_\w(\t_x)\\
	&=\sum_{j\leq 0}\Pi_{\nu_{j},-1}\sum_{x=\nu_{j-1}+1}^{\nu_j}\Pi_{x,\nu_j-1}E_\w[\t_x]\Var_\w(\t_x) \\
        &\leq \sum_{j\leq 0}\Pi_{\nu_{j},-1}\sum_{x=\nu_{j-1}+1}^{\nu_{j}}E_\w[\t_{\nu_j}]\Var_\w(\t_x) \\
        &=\sum_{j\leq 0}\Pi_{\nu_{j},-1}E_\w[\t_{\nu_j}]\sigma_{j,\w}^2,
	\end{align*}
where in the inequality in the third line we used that $\Pi_{x,y-1}E_\w[\t_x] \leq E_\w[\t_y]$ for any $x<y$ (which follows easily from \eqref{QE}). 
Applying this to \eqref{recursive} and taking expectations under the measure $Q$ we obtain 
	\begin{align*}
	E_Q[E_\w[\t_{0}^3]]&\leq E_Q[E_\w[\t_{0}]]+6\sum_{j\leq 0} E_Q\left[\Pi_{\nu_j,-1}E_\w[\t_{\nu_j}]\sigma_{j,\w}^2\right]\\
%	&= E_Q[E_\w[\t_{0}]]+6\sum_{k=0}^\infty E_Q\left[\Pi_{0,\nu_1-1}\right]^k E_Q\left[E_\w[\t_{\nu_1}]\sigma_{1,\w}^2\right]\\
	&=E_Q[E_\w[\t_{0}]]+6E_Q\left[E_\w[\t_{\nu_1}]\sigma_{1,\w}^2\right]\sum_{k=0}^\infty E_Q\left[\Pi_{0,\nu_1-1}\right]^k, 
	\end{align*}
where in the last equality we used that the blocks of environments between ladder locations are i.i.d.\ under the measure $Q$. 
	Since $E_Q[\Pi_{0,\nu_1-1}]<1$, the infinite sum on the right is finite, and thus we need only to show that 
the other terms in the last line above are finite. 
To this end, first note that it was shown in \cite[Lemma 2.2 and Theorem 5.1]{pzSL1} that $Q(E_\w[\tau_0]>x) \leq C e^{-cx}$ and $Q(\s_{1,\w}^2 > x) \sim C x^{-\k/2}$. 
This implies that $E_Q[E_\w[\tau_0]]<\infty$, and since $E_\w[\tau_{\nu_1}]$ has the same distribution as $E_\w[\tau_0]$ under $Q$ by applying H\"older's inequality for some $p \in (1,\k/2)$ we have
\[
 E_Q\left[E_\w[\t_{\nu_1}]\sigma_{1,\w}^2\right]\leq E_Q\left[E_\w[\t_0]^{p/(p-1)}\right]^{(p-1)/p} E_Q\left[\left(\sigma_{1,\w}^2\right)^p\right]^{1/p} < \infty. 
\]
This completes the proof of \eqref{tau03m}
\end{proof}

Applying Lemma \ref{lem:nonmain} to \eqref{decomp:thirdmoment}, we see that to finish the proof of Proposition \ref{prop:uppermain2} we need to prove the following 
\begin{lem}\label{liminf:Q}
If $\k \in (2,3)$, then
 \[
  \liminf_{n\to\infty} \frac{1}{n^{3/\k}} \sum_{i=1}^n \mu_{i,\w}^3 = 0, \quad Q\text{-a.s.}
 \]
\end{lem}

To prepare for the proof of Lemma \ref{liminf:Q} we first prove the following. 
	\begin{lem}\label{limit:stable}
If $\k \in (0,3)$, then under the measure $Q$ on environments, 
		\begin{equation}\label{stablelim}
		\frac{1}{n^{3/\k}} \sum_{i=1}^n \mu_{i,\w}^3 \Longrightarrow Y_{\kappa/3}, \quad \text{as } n\to \infty,
		\end{equation}
		where $Y_{\k/3}$ is a totally asymmetric $(\k/3)$-stable random variable. 
		In particular, this implies that for any $\e>0$,
		\begin{equation}\label{limit:zero}
		\liminf_{n\to\infty} Q\left( \sum_{i=1}^n \mu_{i,\w}^3 < \e n^{3/\k} \right) > 0. 
		\end{equation}
	\end{lem}
	
	\begin{proof}
		It was shown in \cite[Proposition 5.1]{psWQLTn} that 
		$N_{n,\w} = \sum_{i=0}^{n-1} \d_{\mu_{i,\w}/n^{1/\k}} \Longrightarrow N$ as $n\to\infty$ under the measure $Q$ where $N$ is a non-homogeneous Poisson process on $(0,\infty)$ with intensity $C x^{-\k-1}$. 
%From this the proof of \eqref{stablelim} is standard, though we will sketch out the details for the reader. 
Now, consider the functional $\phi$ on the space $\mathcal{M}_p$ of point processes on $(0,\infty]$ defined by $\phi(\sum_i \delta_{x_i}) = \sum_i x_i^3$. Then since $\k < 3$ it is standard that $\phi(N)$ is a $(\k/3)$-stable random variable. 
Thus, \eqref{stablelim} is the statement that $\phi(N_{n,\w}) \Rightarrow \phi(N)$. 
However, this doesn't follow immediately from $N_{n,\w} \Rightarrow N$ since the functional $\phi$ is not continuous with respect to the vague topology on $\mathcal{M}_p$. 
To fix this we do a truncation. That is, for any $\delta>0$ let $\phi_\delta$ be the functional on $\mathcal{M}_p$ defined by $\phi_\delta( \sum_i \delta_{x_i}) = \sum_i x_i^3 \ind{x_i > \d}$. The functional $\phi_\d$ is continuous on the set of point processes with no atoms at $\{\delta\}$ and since the Poisson point process $N$ belongs to this set with probability one, the continuous mapping theorem implies that 
$\phi_\d(N_{n,\w}) \Rightarrow \phi_\d(N)$. Since $\phi_\d(N) \to \phi(N)$ as $\d \to 0$, to complete the proof that $\phi(N_{n,\w}) \Rightarrow \Phi(N)$ we need only to show that 
\begin{equation}\label{trunc}
 \lim_{\d\to 0} \limsup_{n\to\infty} Q\left( |\phi(N_{n,\w}) - \phi_\d(N_{n,\w}) | > \e \right) = 0, \quad \forall \e>0. 
\end{equation}
To this end, note that 
\begin{align*}
 Q\left( |\phi(N_{n,\w}) - \phi_\d(N_{n,\w}) | > \e \right)
&= Q\left( \sum_{i=1}^n \mu_{i,\w}^3 \ind{\mu_{i,\w} \leq \d n^{1/\k}} > \e n^{3/\k} \right)  \\
&\leq \frac{1}{\e n^{3/\k}} n E_Q[ \mu_{1,\w}^3 \ind{\mu_{1,\w} \leq \d n^{1/\k}}]. 
\end{align*}
The tail asymptotics of $\mu_{1,\w}$ under $Q$ given in \eqref{Qbtails} imply that $ E_Q[ \mu_{1,\w}^3 \ind{\mu_{1,\w} \leq \d n^{1/\k}}] \sim \frac{K_\infty \k}{3-\k} n^{\frac{3}{\k}-1} \d^{3-\k}$. 
This is enough to imply \eqref{trunc} which completes the proof of the Lemma. 
\end{proof}

\begin{proof}[Proof of Lemma \ref{liminf:Q}]
We begin by briefly sketching how the proof will follow from Lemma \ref{limit:stable}. For any $\e>0$ and $n\geq 1$ let 
\begin{equation}\label{Bendef}
 B_{\e,n} = \left\{ \sum_{i=1}^n \mu_{i,\w}^3 < \e n^{3/\k} \right\}. 
\end{equation}
It follows from Lemma \ref{limit:stable} that $Q(B_{\e,n})$ is bounded away from zero for all $n$ large. Moreover, if $n' \gg n$ then the events $B_{\e,n'}$ and $B_{\e,n}$ are very weakly dependent and so we should expect that infinitely many of the events $B_{\e,n}$ occur for every $\e>0$. 

To make this argument precise we need to create some additional independence by modifying the environment in certain locations. 
For any $m \in \Z$ let $\w(m)$ be the environment $\w$ modified by placing a right directed reflection point at location $m$.  That is,
	\[\w(m)_x=\begin{cases}
	\w_x&,\quad x \neq m\\
	1&,\quad x=m. 
	\end{cases}\] 
Now define $\tilde{\mu}_{i,\w}^{(n)}$ to be a quenched expected crossing time from $\nu_i$ to $\nu_{i+1}$ with a reflection point located at $\nu_{i-1}-\fl{\sqrt{n}}-1$; that is,
	\[\tilde{\mu}_{i,\w}^{(n)}:=E_{\omega(\nu_{i-1}-\fl{\sqrt{n}}-1)}\left[T^{\nu_{i-1}}_{\nu_i}\right].\]
%The advantage of these modified quenched expectations is that $\tilde{\mu}_{i,\w}^{(n)}$ and $\tilde{\mu}_{j,\w}^{(n)}$ are independent if $i-j > \fl{\sqrt{n} }$. 
We first prove that there is not much difference between $\tilde{\mu}_{i,\w}^{(n)}$ and $\mu_{i,\w}^3$ in the sense that 
\begin{equation}\label{dif3}
 \lim_{n\to\infty} \frac{1}{n^{3/\k}} \sum_{i=1}^n \left( \mu_{i,\w}^3 - (\tilde\mu_{i,\w}^{(n)})^3 \right) = 0, \qquad Q\text{-a.s.}
\end{equation}
To see this, note first of all that $\mu_{i,\w} \geq \tilde\mu_{i,\w}^{(n)}$ and thus 
\[
 \left( \mu_{i,\w}^3 - (\tilde\mu_{i,\w}^{(n)})^3 \right) \leq 3 \mu_{i,\w}^2  \left( \mu_{i,\w} - \tilde\mu_{i,\w}^{(n)} \right). 
\]
Secondly, it can be shown that the added reflection in $\tilde\mu_{i,\w}^{(n)}$ is far enough to the left to not change much. In fact, it was shown in \cite[Lemma 5.3]{psWQLTn} that $Q( \mu_{i,\w} - \tilde{\mu}_{i,\w}^{(n)} > e^{-n^{1/4}} ) \leq C e^{-c \sqrt{n}}$ so that the Borel-Cantelli implies that 
 $Q$-a.s., for $n$ large enough, we have $\mu_{i,\w} - \mu_{i,\w}^{(n)} \leq e^{-n^{1/4}}$ for all $1\leq i\leq n$. 
Thus, we have that $Q$-a.s.\ for $n$ large enough that  
\[
\frac{1}{n^{3/\k}} \sum_{i=1}^n \left( \mu_{i,\w}^3 - (\tilde\mu_{i,\w}^{(n)})^3 \right)
\leq e^{-n^{1/4}} n^{1-3/\k} \left( \frac{1}{n} \sum_{i=1}^n  \mu_{i,\w}^2 \right) 
\]
It follows from Birkhoff's Ergodic Theorem (since $\k>2$) that the term in parenthesis on the right has a finite limit as $n\to\infty$. Since $e^{-n^{1/4}} n^{1-3/\k} \to 0$ this completes the proof of \eqref{dif3}. 

Due to \eqref{dif3}, to complete the proof of Proposition \ref{liminf:Q} we need only to show that 
\begin{equation}\label{tmu3lim}
 \liminf_{n\to\infty} \frac{1}{n^{3/\k}} \sum_{i=1}^n \left( \tilde\mu_{i,\w}^{(n)} \right)^3 = 0, \quad Q\text{-a.s.}
\end{equation}

In fact, we will show that the liminf is zero along the subsequence $n_k = 2^{2^k}$. 
To this end, for any fixed $\e>0$ and $k\geq 1$ let 
\[
 A_{\e,k}^1 = \left\{  \sum_{i=2n_{k-1} + 1}^{n_k} \left( \tilde\mu_{i,\w}^{(n_k)} \right)^3 < \e n_k^{3/\k} \right\} \text{ and } A_{\e,k}^2 = \left\{  \sum_{i= 1}^{2n_{k-1}} \left( \tilde\mu_{i,\w}^{(n_k)} \right)^3 < \e n_k^{3/\k} \right\}
\]
Due to the reflections added in the definition of $\tilde\mu_{i,\w}^{(n)}$ and the fact that $\sqrt{n_k} = n_{k-1}$, it follows that the event $A_{\e,k}^1$ depends only on the environment in the interval 
$[\nu_{2n_{k-1}}-n_{k-1}, \nu_{n_k}) \subset [\nu_{n_{k-1}}, \nu_{n_k} )$.
Thus, the sequence of events $\{A_{\e,k}^1 \}_{k\geq 1}$ are independent. 
Now, since the event $A_{\e,k}^1 \supset B_{\e,n_k}$ defined in \eqref{Bendef}, it follows from Lemma \ref{limit:stable}
that $\liminf_{k\to\infty} Q(A_{\e,k}^1) \geq \lim_{k\to \infty} Q(B_{\e,n_k}) > 0$.  
Since the events $A_{\e,k}^1$ are independent and have probabilities that are uniformly bounded away from zero we have that with probability 1 infinitely many of the events $A_{\e,k}^1$ occur. 
%{\color{red}
%That is, 
%\begin{equation}\label{tmu3liminf}
% \liminf_{k\to\infty} \frac{1}{n^{3/\k}} \sum_{i=2n_{k-1} + 1}^{n_k} \left( \tilde\mu_{i,\w}^{(n_k)} \right)^3 \leq \e, \quad Q\text{-a.s.}
%\end{equation}
%}
In regards to the event $A_{\e,k}^2$, 
note that since $n_{k-1}=n_k^{1/2}$ we have 
\[
 Q\left((A_{\e,k}^2)^c \right)\leq Q\left(\exists i\in[1,2n_{k-1}]:(\tilde\mu_{i,\w}^{(n_k)})^3 > \frac{\e n^{3/\k}}{2 n_{k-1}} \right)\leq 2n_k^{1/2} Q\left( \mu_{1,\w}^3 > \frac{\e}{2} n_k^{\frac{3}{\k}-\frac{1}{2} }\right) \leq C \e^{-\k/3} n_k^{-\frac{1}{2}+\frac{\k}{6}},
\]
for some $C<\infty$. Since $\k<3$ implies that the right side is summable in $k$, then the 
Borel-Cantelli Lemma implies that $Q$-a.s., all but finitely many of the events $A_{\e,k}^2$ occur. 
Thus, we can conclude that $Q$-a.s.\ infinitely many of the events $A_{\e,\k}^1 \cap A_{\e,\k}^2$ occur. 
Since this is true for any $\e>0$ this is enough to prove \eqref{tmu3lim}.
\end{proof}

\subsection{Upper bound when \texorpdfstring{$\k=3$}{kappa=3}}

We will now show how the argument above can be adapted to the case $\k=3$ to prove the upper bound in Theorem \ref{th:k3exact}. That is, we will show that when $\k=3$ there exists a constant $u >0$ such that 
\[
 \lim_{n\to\infty} P\left( \frac{\sqrt{n}}{\log n} \| F_{n,\w} - \Phi \|_\infty \leq u \right) = 1.
\]
%Many of the steps are the same as in $\k \in (2,3)$. 
To begin, it follows from \eqref{bound1} and \eqref{qVarlim} that we need only to show that 
\begin{equation}\label{k3suff2}
 \lim_{n\to\infty} P\left( \sum_{k=1}^n E_\w[\tau_k^3] \leq b n \log n \right) = 1, \quad \text{for some $b<\infty$.}
\end{equation}

We claim that \eqref{k3suff2} is implied by the same limit holding for the measure $Q$ on environments; that is, 
\begin{equation}\label{k3suff3}
  \lim_{n\to\infty} Q\left( \sum_{k=1}^{\nu_n} E_\w[\tau_k^3] \leq b n \log n \right)=1, \quad \text{for some } b<\infty. 
\end{equation}
%(NEED TO JUSTIFY THIS. It doesn't seem as easy to pass through $\tilde{P}$ using that $\tilde{P} \ll Q$. I think I have an argument that works though. I can add it later.)
That \eqref{k3suff3} implies \eqref{k3suff2} is justfied in a similar manner to the argument immediately following Proposition \ref{prop:uppermain2}. 
Indeed, $ \lim_{n\to\infty} Q\left( \sum_{k=1}^{\nu_n} E_\w[\tau_k^3] \leq b n \log n \right)=1$ implies that $\lim_{n\to\infty} \tilde{P}\left( \sum_{k=1}^{\nu_n} E_\w[\tau_k^3] \leq b n \log n \right)=1$ since $\tilde{P} \ll Q$, and this in turn implies that $\lim_{n\to\infty} P\left( \sum_{k=1}^{\nu_n} E_\w[\tau_k^3] \leq b n \log n \right)=1$ since we have $\sum_{k=1}^{\nu_n} E_\w[\tau_k^3] \leq \sum_{k=\nu_0 + 1}^{\nu_n} E_\w[\tau_k^3]$. Finally \eqref{k3suff2} follows for some $b<\infty$ since $\nu_n/n \to \bar{\nu}$, $P$-a.s.
To prove \eqref{k3suff3}, we use \eqref{decomp:thirdmoment} to bound $\sum_{k=1}^{\nu_n} E_\w[\tau_k^3]$ by $6 \sum_{i=1}^n \mu_{i,\w}^3$ plus some error terms which by Lemma \ref{lem:nonmain} are $o(n\log n)$, $Q$-a.s. 
Therefore, \eqref{k3suff3} will follow from
\begin{equation}\label{k3suff4}
 \lim_{n\to\infty} Q\left( \sum_{i=1}^n \mu_{i,\w}^3 \leq b n \log n \right) = 1, \quad \text{for some } b<\infty. 
\end{equation}
Let $r_n = n \sqrt{\log n}$ and note that \eqref{Qbtails} implies that since $\lim_{n\to\infty} Q\left( \max_{i\leq n} \mu_{i,\w}^3 > r_n \right) = 0$ then 
\eqref{k3suff3} will follow if we show that
\begin{equation}\label{k3suff5}
 \lim_{n\to\infty} Q\left( \sum_{i=1}^n \left( \mu_{i,\w}^3 \wedge r_n \right) > b n \log n \right) = 0, \quad \text{for some } b<\infty.
\end{equation} 
To show this, we will use that \eqref{Qbtails} implies that 
\begin{equation}\label{truncmv}
 E_Q[\mu_{1,\w}^3 \wedge r_n ] \sim K_\infty \log n \quad\text{and}\quad \Var_Q(\mu_{1,\w}^3 \wedge r_n) \sim 2K_\infty n \sqrt{\log n} \quad \text{as } n\to\infty
\end{equation}
and we will also use the following covariance bound. 
\begin{lem}\label{lem:covbound}
 If $\k=3$ and $r_n = n \sqrt{\log n}$, then there exist constants $C,c_0>0$ such that for $n$ sufficiently large 
\[
 \Cov_Q\left( \mu_{1,\w}^3 \wedge r_n, \mu_{1+k,\w}^3 \wedge r_n \right) \leq C   n^{3/2} (\log n)^{3/4} e^{-c_0 k}, \quad \forall k\geq 1. 
\]
\end{lem}
\begin{rem}
We note that since the sequence $\mu_{i,\w}^3 \wedge r_n$ is stationary under $Q$, it follows from the variance asymptotics in \eqref{truncmv} that an easy bound on the covariances in Lemma \ref{lem:covbound} is $C n\sqrt{\log n}$.
The bound in Lemma \ref{lem:covbound} is only significantly better than this trivial bound for $k \gg \frac{1}{2c_0} \log n$. 
\end{rem}

Postponing the proof of Lemma \ref{lem:covbound} for the moment we show how to prove \eqref{k3suff5}. 
If $b>K_\infty$ we have that for $n$ sufficiently large %and any fixed $A>0$ that 
\begin{align*}
&Q\left( \sum_{i=1}^n \left( \mu_{i,\w}^3 \wedge r_n \right) > b n \log n \right)\\
&\leq Q\left( \sum_{i=1}^n \left\{ \left( \mu_{i,\w}^3 \wedge r_n  \right) - n E_Q[\mu_{1,\w}^3 \wedge r_n ] \right\}  > \frac{b-K_\infty}{2} n \log n \right) \\
&\leq \frac{4}{(b-K_\infty)^2 n^2 \log^2 n} \Var_Q\left(  \sum_{i=1}^n \left( \mu_{i,\w}^3 \wedge r_n  \right) \right) \\
%&=\frac{4}{(b-K_\infty)^2 n^2 \log^2 n} \left\{ n \Var_Q( \mu_{1,\w}^3 \wedge r_n ) + 2 \sum_{k=1}^{n-1} (n-k) \Cov_Q\left( \mu_{1,\w}^3 \wedge r_n, \mu_{1+k,\w}^3 \wedge r_n \right) \right\} \\
&\leq \frac{4}{(b-K_\infty)^2 n^2 \log^2 n} \left\{ \frac{1}{c_0}n\log n \Var_Q( \mu_{1,\w}^3 \wedge r_n ) + 2n \sum_{k=\fl{\frac{1}{2c_0} \log n} }^{n-1} \Cov_Q\left( \mu_{1,\w}^3 \wedge r_n, \mu_{1+k,\w}^3 \wedge r_n \right) \right\} \\
%&= \bigo\left( \frac{1}{\sqrt{\log n}} \right) + \bigo\left( \frac{1}{n\log^2 n}\right)  \sum_{k=\fl{A \log n} }^{n-1} \Cov_Q\left( \mu_{1,\w}^3 \wedge r_n, \mu_{1+k,\w}^3 \wedge r_n \right) \\
%&= \bigo\left( \frac{1}{\sqrt{\log n}} \right) + \bigo\left( \frac{n^{\frac{1}{2}-A c_0}}{(\log n)^{5/4}}\right),  
&= \bigo\left( \frac{1}{\sqrt{\log n}} \right) + \bigo\left( \frac{1}{(\log n)^{5/4}}\right), 
\end{align*}
where the second to last line follows from stationarity under $Q$ and the last line follows from \eqref{truncmv} and Lemma \ref{lem:covbound}. Thus, we have shown, pending the proof of Lemma \ref{lem:covbound}, that the limit in \eqref{k3suff5} holds for any $b>K_\infty$. 

\begin{proof}[Proof of Lemma \ref{lem:covbound}]
 Note that for $k\geq 1$
\begin{align*}
 \mu_{k+1,\w} &= \nu_{k+1}-\nu_k + 2 \sum_{j=\nu_k}^{\nu_{k+1}-1} W_j \\
&= \nu_{k+1}-\nu_k + 2 \sum_{j=\nu_k}^{\nu_{k+1}-1} (W_{\nu_1+1,j} + \Pi_{\nu_1+1,j} W_{\nu_1} ) \\
&=: \hat\mu_{k+1,\w} + 2 W_{\nu_1} \Pi_{\nu_1+1,\nu_k-1} R_{\nu_k,\nu_{k+1}-1}, 
\end{align*}
where $\hat\mu_{k+1,\w}$ is independent of $\s(\w_x, \, x\leq  \nu_1)$ and therefore independent of $\mu_{1,\w}$. 
(Note that we can interpret $\hat\mu_{k+1,\w}$ as the expected crossing time from $\nu_k$ to $\nu_{k+1}$ when the environment $\w$ is modified by adding a reflection at $\nu_1$.)
The following lemma shows that $\hat\mu_{k+1,\w}$ is very close to $\mu_{k+1,\w}$ when $k$ is large. 
\begin{lem}\label{lem:truncdiff}
 There exist constants $c_1,c_2,c_3>0$ such that 
\[
 Q\left( \mu_{i,\w}^3 - \hat\mu_{k+1,\w}^3 > e^{-c_1 k} \right)
\leq c_2 e^{-c_3 k}. 
\]
\end{lem}
\begin{proof}
 Since  $\mu_{k+1,\w}^3 - \hat\w_{k+1,\w}^3 \leq 3(\mu_{k+1,\w} - \hat\mu_{k+1,\w} ) \mu_{k+1,\w}^2$ we have that 
\begin{align}
 Q\left(  \mu_{k+1,\w}^3 - \hat\mu_{k+1,\w}^3 > e^{-c_1 k} \right) 
&\leq Q\left( \mu_{k+1,\w} - \hat\mu_{k+1,\w} > e^{-2c_1 k} \right) + Q\left( 3 \mu_{1,\w}^2 > e^{c_1 k} \right) \nonumber  \\
&\leq E_Q\left[ \mu_{k+1,\w} - \hat\mu_{k+1,\w} \right] e^{2c_1 k} + C e^{-(c_1\k/2) k}. \label{m3hm3tail} 
\end{align}
Now, it follows from the definition of $\hat\mu_{k+1,\w}$ and the fact that the blocks of the environment between ladder locations are i.i.d.\ under $Q$ that 
\[
 E_Q\left[ \mu_{k+1,\w} - \hat\mu_{k+1,\w} \right] = E_Q[W_0]E_Q[\Pi_{1,\nu_1-1}] E_Q[\Pi_{0,\nu_1-1}]^{k-2} E[R_{0,\nu_1-1}]
\leq E_Q[W_0]E[\Pi_{0,\nu_1-1}]^{k-2} E[R_0].
\]
Since $E[R_0] < \infty$ when $\k>1$, \cite[Lemma 2.2]{pzSL1} implies $E_Q[W_0]$, and $E_Q[\Pi_{0,\nu_1-1}] < 1$ by the definition of the ladder locations, it follows that if we choose $c_1>0$ small enough then \eqref{m3hm3tail} decreases exponentially in $k$. 
\end{proof}

To control the covariance note that
\begin{align*}
\mu_{k+1,\w}^3 \wedge r_n
&\leq  ( \hat\mu_{k+1,\w}^3\wedge r_n ) + e^{-c_1 k}  + r_n \ind{ \mu_{k+1,\w}^3 - \hat\mu_{k+1,\w}^3 > e^{-c_1 k} },
\end{align*}
and since $\hat\mu_{k+1,\w}$ is independent of $\mu_{1,\w}$, it follows that 
\begin{align*}
& E_Q[(\mu_{1,\w}^3 \wedge r_n)(\mu_{k+1,\w}^3 \wedge r_n)]\\
&\leq E_Q[ \mu_{1,\w}^3\wedge r_n]\left( E_Q[\hat\mu_{k+1,\w}^3\wedge r_n] + e^{-c_1 k} \right) 
+ r_n E_Q\left[ (\mu_{1,\w}^3 \wedge r_n) \ind{ \mu_{k+1,\w}^3 - \hat\mu_{k+1,\w}^3 > e^{-c_1 k} } \right] \\
&\leq   E_Q[ \mu_{1,\w}^3 \wedge r_n]\left( E_Q[ \mu_{1,\w}^3\wedge r_n ] + e^{-c_1 k} \right) + r_n \sqrt{ E_Q[  (\mu_{1,\w}^3 \wedge r_n)^2 ] } \sqrt{ Q\left( \mu_{k+1,\w}^3 - \hat\mu_{k+1,\w}^3 > e^{-c_1 k}  \right) }\\
%&\leq E_Q[ \mu_{1,\w}\wedge r_n]\left( E_Q[ \mu_{1,\w}\wedge r_n] + e^{-c_1 k} \right) + r_n \sqrt{ E_Q[  (\mu_{1,\w}^3 \wedge r_n)^2 ] } \sqrt{c_2} e^{-(c_3/2) k} \\
&\leq E_Q[ \mu_{1,\w}^3\wedge r_n]^2 + C \log n e^{-c_1 k} + C n^{3/2} (\log n)^{3/4} e^{-(c_3/2) k},
\end{align*}
where in the last inequality we used \eqref{truncmv} and Lemma \ref{lem:truncdiff}. 
Thus, it follows that  
\[
 \Cov_Q( \mu_{1,\w}\wedge r_n, \, \mu_{k+1,\w} \wedge r_n ) \leq C n^{3/2} (\log n)^{3/4} e^{-c_0 k}, \quad \forall k\geq 1,
\]
for some $C,c_0>0$ and $n$ large enough. 
\end{proof}

\section{Lower bound on rates of convergence}

The starting point for our lower bound is the following Theorem for Normal approximation of sums of independent random variables. 
The precise statement in this Theorem is new, though as will be seen in the proof it arises easily from the proof of a similar lower bound in \cite{hbReversing}. 

\begin{thm}\label{th:NCLTlb}
 Let $\xi_1,\xi_2,...,\xi_n$ be independent random variables which have zero means and finite 5-th moments; that is $E[|\xi_i|^5] < \infty$ for all $i$. Assume that $\sum_{i=1}^n E[\xi_i^2] = 1$. 
Then, there exists an absolute constant $C>0$ such that 
\[
 C\left(\sup_z|P(W\leq z)-\Phi(z)|+\sum_{i=1}^n E[\xi_i^2]^2\right)
\geq \frac{1}{4} \left| \sum_{i=1}^n E\left[ \xi_i^3 \right] \right| - \frac{1}{32} \sum_{i=1}^n E[|\xi_i|^5]. 
\]
\end{thm}

\begin{proof}
 It follows from computations in \cite{hbReversing} that there exists an absolute constant $C>0$ such that\footnote{To obtain \eqref{lb1}, apply equation (7) in \cite{hbReversing} with the choice of the function $f(w) = e^{-w^2/2}$.}
\begin{equation}\label{lb1}
 C\left(\sup_z|P(W\leq z)-\Phi(z)|+\sum_{i=1}^n E[\xi_i^2]^2\right)
\geq \left| \sum_{i=1}^n E\left[ \xi_i(1-e^{-\xi_i^2/4}) \right] \right|. 
\end{equation}
Since $1-\frac{x^2}{4} \leq e^{-x^2/4} \leq 1-\frac{x^2}{4} + \frac{x^4}{32}$, it follows that 
\[
 \left| x(1-e^{-x^2/4}) - \frac{1}{4} x^3 \right| \leq \frac{1}{32} |x|^5, \qquad \forall x\in\R. 
\]
Applying this to \eqref{lb1} we obtain that 
\begin{align*}
  C\left(\sup_z|P(W\leq z)-\Phi(z)|+\sum_{i=1}^n E[\xi_i^2]^2\right)
&\geq \left|\sum_{i=1}^n \frac{1}{4} E[\xi_i^3]  \right| - \sum_{i=1}^n E\left[ \left| \xi_i(1-e^{-\xi_i^2/4}) - \frac{1}{4} \xi_i^3 \right| \right] \\
&\geq  \frac{1}{4} \left|\sum_{i=1}^n E[\xi_i^3]  \right| - \frac{1}{32} \sum_{i=1}^n E[|\xi_i|^5]. 
\end{align*}

\end{proof}

Applying Theorem \ref{th:NCLTlb} to the random variables $\xi_i = \frac{\tau_i - E_\w[\tau_i]}{\sqrt{\Var_\w(T_n)}}$ gives 
\begin{align}
 \| F_{n,\w} - \Phi \|_\infty 
&\geq \frac{C}{\Var_\w(T_n)^{3/2}} \sum_{i=1}^n E_\w\left[ (\tau_i - E_\w[\tau_i])^3 \right]  \label{lbmain} \\
&\qquad  - \frac{C'}{\Var_\w(T_n)^2}  \sum_{i=1}^n \left( \Var_\w(\tau_i) \right)^2 - \frac{C'}{\Var_\w(T_n)^{5/2}} \sum_{i=1}^n E_\w\left[ |\tau_i - E_\w[\tau_i]|^5 \right]. \label{lbother}
\end{align}
We will see below that the term on the right in \eqref{lbmain} is the main term and that the two terms in \eqref{lbother} are negligible on the scale of $n^{\frac{3}{\k} - \frac{3}{2} }$. 
We begin by showing that the terms in \eqref{lbother} are negligible. 
\begin{lem}\label{lem:45small}
 If $\k >2$ then 
\[
  \lim_{n\to\infty} n^{\frac{3}{2}-\frac{3}{\k}} \left( \frac{1}{\Var_\w(T_n)^2}  \sum_{i=1}^n \left( \Var_\w(\tau_i) \right)^2 + \frac{1}{\Var_\w(T_n)^{5/2}} \sum_{i=1}^n E_\w\left[ |\tau_i - E_\w[\tau_i]|^5 \right] \right) = 0, \quad P\text{-a.s.}
\]
%\[
% \lim_{n\to\infty} \frac{n^{\frac{3}{2}-\frac{3}{\k}} }{\Var_\w(T_n)^2}  \sum_{i=1}^n \left( \Var_\w(\tau_i) \right)^2 = 0, \quad P\text{-a.s.}
%\]
%and 
%\[
% \lim_{n\to\infty} \frac{n^{\frac{3}{2}-\frac{3}{\k}} }{\Var_\w(T_n)^{5/2}} \sum_{i=1}^n E_\w\left[ |\tau_i - E_\w[\tau_i]|^5 \right] = 0, \quad P\text{-a.s.}
%\]
\end{lem}

\begin{proof}
 Since \eqref{qVarlim} implies that  $\Var_\w(T_n)$ grows linearly, $P$-a.s., it is enough to show that 
\begin{equation}\label{var2}
 \lim_{n\to\infty} \frac{1}{n^{\frac{1}{2}+\frac{3}{\k}}} \sum_{i=1}^n  \left( \Var_\w(\tau_i) \right)^2 = 0, \quad P\text{-a.s.},
\end{equation}
and 
\begin{equation}\label{5msum}
  \lim_{n\to\infty} \frac{1}{n^{1+\frac{3}{\k}}} \sum_{i=1}^n  E_\w\left[ |\tau_i - E_\w[\tau_i]|^5 \right]   = 0, \quad P\text{-a.s.}
\end{equation}
To control the sums in \eqref{lbother} we will use the following simple Lemma.
% which we learned from Gennady Sammorodnitsky (personal communication). 
\begin{lem}\label{lem:ergsums}
%Suppose that $\{\zeta_i\}_{i\geq 1}$ is a stationary and ergodic sequence of random variables with $E[|\zeta_i|^p] < \infty$ for some $p \in (0,1]$. Then, 
%\[
%\limsup_{n\to\infty} \left| \frac{1}{n^{1/p}} \sum_{i=1}^n \zeta_i \right| \leq \left( E[|\zeta_1|^p] \right)^{1/p}. 
%\]
%In particular, $\sum_{i=1}^n \zeta_i = \bigo( n^{1/p} )$. 
Suppose that $\{\zeta_i\}_{i\geq 1}$ is a stationary and ergodic sequence of random variables. Let $\a \in (0,1]$ and suppose that $E[|\zeta_i|^p] < \infty$ for all $p \in (0,\a)$. Then,
for any $\gamma > \frac{1}{\a}$, 
\[
\lim_{n\to\infty}  \frac{1}{n^{\gamma}} \sum_{i=1}^n \zeta_i = 0, \quad P\text{-a.s.}
\]
%That is, $\sum_{i=1}^n \zeta_i = o(n^{\frac{1}{\a}+\e})$ for any $\e>0$. 
\end{lem}
\begin{proof}
If $p \in (0,1)$, then 
\[
  \left| \frac{1}{n^{1/p}} \sum_{i=1}^n \zeta_i \right|
  \leq  \frac{1}{n^{1/p}} \sum_{i=1}^n \left|\zeta_i \right|
  = \left(\frac{1}{n} \left(\sum_{i=1}^n |\zeta_i|\right)^p\right)^{1/p}
  \leq \left(\frac{1}{n} \sum_{i=1}^n |\zeta_i|^p\right)^{1/p}.
\]
Since $\{\zeta_i\}_{i\geq 1}$ is a stationary and ergodic sequence, if $p<\a$ then the last term on the right converges to $\left( E[|\zeta_1|^p] \right)^{1/p} < \infty$ as $n\to\infty$.
Thus, with probability 1 the sum $\sum_{i=1}^n \zeta_i$ grows no faster than $n^{1/p}$ for any $p<\a$. 
If $\gamma > \frac{1}{\a}$ then the conclusion of the lemma follows by choosing $p \in (1/\gamma,\a)$. 
\end{proof}
Note that the terms inside the sums in \eqref{var2} and \eqref{5msum} are ergodic sequences under the measure $P$. 
Therefore, by Lemma \ref{lem:qtmom} we can apply Lemma \ref{lem:ergsums} to the sum in \eqref{var2} with $\a=\frac{\k}{4}$ and to the sum in \eqref{5msum} with $\a=\frac{\k}{5}$. 
Since $\k>2$ implies that $\frac{1}{2} + \frac{3}{\k} > \frac{4}{\k}$ and $1+\frac{3}{\k} > \frac{5}{\k}$ this completes the proof of \eqref{var2} and \eqref{5msum}, and thus also of the lemma. 
\end{proof}

It remains now to give a lower bound on the sum in the right side of \eqref{lbmain}.
Again, using the fact that $\Var_\w(T_n)$ grows linearly, $P$-a.s., the proof of \eqref{LS} in Theorem \ref{th:ExactRates} will follow if we can show
\begin{prop}
\begin{equation}\label{LSc3m}
 \limsup_{n\to\infty} \frac{1}{n^{3/\k}} \sum_{i=1}^n E_\w\left[ (\tau_i - E_\w[\tau_i])^3 \right]  = \infty, \quad P\text{-a.s.}
\end{equation}
%\[
% \limsup_{n\to\infty} n^{\frac{3}{2}-\frac{3}{\k}} \frac{1}{\Var_\w(T_n)^{3/2}} \sum_{i=1}^n E_\w\left[ (\tau_i - E_\w[\tau_i])^3 \right]  = \infty, \quad P\text{-a.s.}
%\]
\end{prop}
 
\begin{proof}
We begin by showing the following simple lower bound for centered third moments. 
\begin{equation}\label{3mlb}
 E_\w[(\tau_i - E_\w[\tau_i])^3 ] \geq 16 W_{i-1}^3. 
\end{equation}
To see this, first we expand the centered third moment at a site $i$.
\begin{align}
E_\w\left[(\t_i-E[\t_i])^3\right]&=E_\w\left[\t_i^3\right]-3E_\w\left[\t_i^2\right]E[\t_i]+3E_\w\left[\t_i\right]E_\w\left[\t_i\right]^2-E_\w\left[\t_i\right]^3\nonumber\\
&=E_\w\left[\t_i^3\right]-3\Var(\t_i)E_\w\left[\t_i\right]-E_\w\left[\t_i\right]^3\label{centered-third-1}
\end{align}
It follows from the recursive equation for quenched third moments in Lemma \ref{lem:thirdmoment} that 
%$E_\w[\tau_i^3] \geq E_\w[\t_i]+6E_\w[\t_i]\Var_\w(\t_i)$. 
\[
 E_\w[\tau_i^3] \geq 6E_\w[\t_i]\Var_\w(\t_i).
\]
Using this in \eqref{centered-third-1}
we get 
\begin{equation}\label{lower:second-1}
E_\w\left[(\t_i-E[\t_i])^3\right]\geq  E_\w[\t_i]\left(3\Var_\w(\t_i)-E_\w[\t_i]^2 \right)
\end{equation}
Since it follows from \eqref{QV} that $\Var_\w(\tau_i) \geq 4 W_{i-1}^2$ and 
from \eqref{QE} that  $E_\w[\t_i] \geq 2W_{i-1}$, we have that \eqref{3mlb} follows. 

To prove \eqref{LSc3m}, it will be enough to show the limsup is infinite along the subsequence of ladder locations. Moreover, since $\nu_n/n \to \bar{\nu}$, $P$-a.s., it is enough to show that 
\begin{equation}\label{ll3msum}
 \limsup_{n\to\infty} \frac{1}{n^{3/\k}} \sum_{i=1}^{\nu_n} E_\w\left[ (\tau_i - E_\w[\tau_i])^3 \right] = \infty, \quad P\text{-a.s.}
\end{equation}
Recall the definition of $M_j$ in \eqref{Midef} as the exponential of the height of the potential between ladder locations $\nu_{j-1}$ and $\nu_j$. 
Then for any fixed $j$ it follows from the lower bound in \eqref{3mlb} that 
\begin{align}
 \sum_{i=\nu_{j-1}+1}^{\nu_j} E_\w\left[ (\tau_i - E_\w[\tau_i])^3 \right]
\geq 16 \sum_{i=\nu_{j-1}+1}^{\nu_j} W_{i-1}^3 
&\geq 16 \sum_{i=\nu_{j-1}+1}^{\nu_j} \Pi_{\nu_{j-1},i-1}^3 \nonumber \\
&\geq 16 \max_{\nu_{j-1} < i \leq \nu_j } \Pi_{\nu_{j-1},i-1}^3
= 16 M_j^3. \label{M3lb}
\end{align}
Therefore, we have that \eqref{ll3msum} will follow if we can show 
\begin{equation}\label{M3sum}
 \limsup_{n\to\infty} \frac{1}{n^{3/\k}} \sum_{j=1}^n M_j^3 = \infty. 
\end{equation}
We noted in \eqref{M} that under the measure $Q$ the random variables  $\{M_j\}_{j\geq 1}$ are i.i.d.\ with tails asymptotic to $C x^{-\k}$. 
Under the measure $P$, the sequence $\{M_j\}_{j\geq 1}$ is still independent and for $j\geq 2$ they all have the same distribution as under $Q$. 
Thus, for the sum in \eqref{M3sum} the terms in the summand are independent and all but the first have tail decay $P(M_j^3 > x) \sim C x^{-\k/3}$. 
From this it is standard to show that \eqref{M3sum} holds. 
First of all, we know that $n^{-3/\k} \sum_{j=1}^n M_j$ converges in distribution as $n\to\infty$ to a $(\k/3)$-stable random variable with support on $(0,\infty)$. Then letting $n_k = 2^{2^k}$ we have that for any $L<\infty$ the sequence of events $\{ \sum_{j=n_{k-1}+1}^{n_k} M_j^3 > L n_k^{3/\k} \}$ are independent with probability bounded away from zero uniformly in $k$. Thus, with probability one, infinitely many of these events hold and so
 $\limsup_{k\to\infty} n_k^{-3/\k} \sum_{j=n_{k-1}+1}^{n_k} M_j^3 \geq L$. 
Since this is true for any $L<\infty$ we have 
\[
 \limsup_{k\to\infty} \frac{1}{n_k^{3/\k}} \sum_{j=n_{k-1}+1}^{n_k} M_j^3 = \infty, 
\]
which implies \eqref{M3sum}. 
\end{proof}

\subsection{Lower bound when \texorpdfstring{$\k=3$}{kappa=3}}

We now prove the lower bound in Theorem \ref{th:k3exact}. That is, if $\k=3$ there exists a constant $\ell >0$ such that 
\begin{equation}\label{lbk3}
 \lim_{n\to\infty} P\left( \frac{\sqrt{n}}{\log n} \| F_{n,\w} - \Phi \|_\infty  \geq \ell \right) = 1.
\end{equation}
Again the proof follows many of the same steps as above for the case $\k \in (2,3)$. 

First of all, since Lemma \ref{lem:45small} holds for $\k=3$ then the terms in \eqref{lbother} are $o(\frac{\log n}{\sqrt{n}})$, $P$-a.s.
Therefore, \eqref{lbk3} will follow if we can show that
\begin{equation}\label{k3suff1}
 \lim_{n\to\infty} P\left( \sum_{k=1}^n E_\w[(\tau_k - E_\w[\tau_k])^3] \geq a n \log n \right)=1, \quad \text{for some } a>0.  
\end{equation}
Using \eqref{M3lb}, we have that 
\begin{align*}
 P\left( \sum_{k=1}^n E_\w[(\tau_k - E_\w[\tau_k])^3] \geq a n \log n \right)
&\geq P\left( \sum_{k=1}^{\nu_{\fl{cn}}} E_\w[(\tau_k - E_\w[\tau_k])^3] \geq a n \log n \right) - P(\nu_{\fl{cn}} > n) \\
&\geq P\left( 16 \sum_{j=1}^{\fl{cn}} M_j^3 \geq a n \log n \right) - P(\nu_{\fl{cn}} > n). 
\end{align*}
If $c< 1/\bar{\nu}$, then the last probability on the right vanishes as $n\to \infty$. On the other hand, since the random variables $M_j$ are i.i.d.\ with $P(M_j^3 > x) \sim C_0/x$, it follows from classical results that $\lim_{n\to\infty} \frac{1}{n\log n} \sum_{j=1}^{n} M_j^3 = C_0$, in $P$-probability. Therefore, by choosing $a$ small enough the first term on the right above tends to 1.

%\newpage 

\appendix
\section{The Quenched third moment of a crossing time}\label{app:3m}
\begin{proof}[Proof of Lemma \ref{lem:thirdmoment}]
	For ease of notation we will introduce the following notation 
	\[e_i = E_\w[\t_i] \text{ and } v_i = \Var_\w(\t_i)\]
	for the quenched mean and variance of $\t_i$ which will be used throughout the proof.  We first derive a recursive equation of $E_\w[\t_1^3]$.  That is
	\begin{equation}\label{appendix:main}
	E_\w[\t_1^3]=\frac{1}{\w_{0}}+6e_1v_1+\r_{0} E_\w[\t_0^3].
	\end{equation}
	Then, we generalize our result to any $\t_i$ for $i\in\mathbb{Z}$, a time to cross from a site $i-1$ to $i$.  Starting at the origin, we can decompose the crossing time $\t_1$ as sums of hitting time to site 1 with the conditions that the first step is either the site 1 or the site -1.  Then, the hitting time to reach the right neighbor site 1 is decomposed to
	\begin{equation}\label{appendix1:t1}
	\t_1=\ind{X_1=1}+\ind{X_1=-1}(1+\t_0+\t_1')=1+\ind{X_1=-1}(\t_0+\t_1')
	\end{equation}
	where $\t'_1$ and $\t_1$ are same distribution but independent to each other.  From \eqref{appendix1:t1}, the quenched expectation of $\t_1$ is
	\[e_1=1+(1-\w_0)(e_0+e_1).
	\]
	Using the notation $\rho_0=(1-\w_0)/\w_0$ and solving for $\r_0 e_0$, the above equation becomes  
	\begin{equation}\label{appendix1:tau1}
	\r_0 e_0=e_1-\frac{1}{\w_0}=e_1-1-\r_0.
	\end{equation}
	Similarly, we obtain $\t_1^2$ from \eqref{appendix1:t1} by
	\begin{align}
	\t_1^2&=(1+\ind{X_1=-1}(\t_0+\t_1'))^2\nonumber\\
	&=1+2\ind{X_1=-1}(\t_0+\t_1')+\ind{X_1=-1}(\t_0+\t_1')^2\nonumber\\
	&=1+2\ind{X_1=-1}(\t_0+\t_1')+\ind{X_1=-1}(\t_0^2+2\t_0\t_1'+\t_1'^2)\nonumber\\
	&=1+\ind{X_1=-1}(2\t_0+2\t_1'+\t_0^2+2\t_0\t_1'+\t_1'^2)\nonumber
	\end{align}
	Since $\{t_i\}_{i\in\mathbb{Z}}$ are independent under the quenched law, taking quenched expectation to both sides of the last equality and solving for $E_\w[\t_1^2]$ yield
	\[
	E_\w[\t_1^2]=\frac{1}{\w_0}+2\left(\r_0 e_0+\r_0 e_1+\r_0 e_0e_1\right)+\rho_0 E_\w[\t_0^2].
	\]
	Then applying the equation in \eqref{appendix1:tau1}, we get
	\begin{align*}
	E_\w[\t_1^2]&=\frac{1}{\w_0}+2\left(e_1-\frac{1}{\w_0}+\r_0 e_1+\left(e_1-1-\r_0\right)e_1\right)+\rho_0 E_\w[\t_0^2]\\
	&=2e_1^2-\frac{1}{\w_0}+\rho_0 E_\w[\t_0^2]
	\end{align*}
	Then using $E_\w[\t_1^2]=v_1+e_1^2$ and solving for $\r_0 E_\w[\t_0^2]$ yields 
	\begin{equation}
	\r_0 E_\w[\t_0^2]=v_1-e_1^2+1+\r_0\label{tau2-1}.
	\end{equation}
	To get the third moment of $\t_1$, we again use the equation in \eqref{appendix1:t1} and expand $\t_1^3$.  That is 
	\begin{align}
	\t_1^3&=(1+\ind{X_1=-1}(\t_0+\t_1'))^3\nonumber\\
	&=1+3\ind{X_1=-1}(\t_0+\t_1')+3\ind{X_1=-1}(\t_0+\t_1')^2+\ind{X_1=-1}(\t_0+\t_1')^3\nonumber\\
	&=1+\ind{X_1=-1}(3\t_0+3\t_1'+3\t_0^2+6\t_0\t_1'+3\t_1'^2+\t_0^3+3\t_0^2\t_1'+3\t_0\t_1'^2+\t_1'^3).\nonumber
	\end{align}
	Then taking a quenched expectation to the last equality and solving for $E_\w[\t_1^3]$ yield
	\begin{align}
	E_\w[\t_1^3]&=\frac{1}{\w_0}+3\left\{\r_0e_0+\r_0e_1+\r_0E_\w[\t_0^2]+2\r_0e_0e_1+\r_0E_\w[\t_1^2]+\r_0E_\w[\t_0^2]e_1+\r_0e_0E_\w[\t_1^2]\right\}+\r_0 E_\w[\t_0^3]\nonumber\\
	&=\frac{1}{\w_0}+3\left\{\r_0e_0+\r_0E_\w[\t_0^2]+\left(\r_0+2\r_0e_0+\r_0E_\w[\t_0^2]\right)e_1+\left(\r_0+\r_0e_0\right)E_\w[\t_1^2]\right\}+\r_0 E_\w[\t_0^3]\label{appendix1:tau3}
	\end{align}
	Then replacing $\r_0 e_0$ and $\r_0 E_\w[\t_0^2]$ by \eqref{appendix1:tau1} and \eqref{tau2-1} and $E_\w[\t_1^2]=v_1+e_1^2$ simplify the second term of  \eqref{appendix1:tau3} as 
	\[e_1+v_1-e_1^2+(2e_1+v_1-e_1^2-1)e_1+(e_1-1)(v_1+e_1^2)=
	2e_1v_1\]
	and we get \eqref{appendix:main}.  In general for a crossing time on any site $i\in\mathbb{Z}$, we obtain the following recursive equation.
	\[E_\w[\t_i^3]=\frac{1}{\w_{i-1}}+6E_\w[\t_i]\Var_\w(\t_i)+\r_{i-1} E_\w[\t_{i-1}^3].\]
\end{proof}

\bibliographystyle{alpha}
%\bibliography{Reference_Advanced}
\bibliography{RWREref}

\end{document}